\documentclass[12pt]{amsart}
\usepackage{amsmath,amsthm,latexsym,amscd,amsbsy,amssymb,amsfonts,fleqn,leqno,
euscript}

\numberwithin{equation}{section}
\newcommand{\HDD[2]}{{#1}\mbox{-}\overline{\mathrm{DD}\hskip-1pt}\hskip1pt{}^{\{#2\}}}
\newcommand{\MAP[2]}{#1\mbox{-}\mathrm{MAP}^{#2}} 

\newcommand{\LC[1]}{\mathrm{LC}^{#1}}
\newcommand{\NDDP[2]}{{#1}\mbox{-}\mathrm{DD}\hskip-1pt\hskip1pt{}^{\{#2\}}}

\newcommand{\U}{\mathcal U}

\newcommand{\e}{\varepsilon}

\newtheorem{thm}{Theorem}[section]
\newtheorem{pro}[thm]{Proposition}
\newtheorem{lem}[thm]{Lemma}
\newtheorem{cor}[thm]{Corollary}

\begin{document}

\title[Parametric set-wise injective maps]
{Parametric set-wise injective maps}
\author{Eiichi Matsuhashi}
\address{Department of Mathematics and Computer Science, Shimane University,
Matsue, Shimane, 690-8504, Japan}
\email{matsuhashi@riko.shimane-u.ac.jp}
\author{Vesko Valov}
\address{Department of Computer Science and Mathematics,
Nipissing University, 100 College Drive, P.O. Box 5002, North Bay,
ON, P1B 8L7, Canada}
\email{veskov@nipissingu.ca}
\thanks{} 

\begin{abstract}
We introduce the notion of set-wise injective maps and provide results about fiber embeddings. Our results improve some
previous results in this area.
\end{abstract}
\maketitle \markboth{}{Set-wise injective maps}


\section{Introduction}
All spaces in the paper are assumed to be metrizable and all maps
continuous. Unless stated otherwise, any function space $C(X,M)$ is
endowed with the {\em source limitation topology}. This topology,
known also as the {\em fine topology}, was introduced in
\cite{w} and has a base at a given $g\in C(X,M)$ consisting of the
sets
$$B_\varrho(g,\e)=\{h\in C(X,M):\varrho(h,g)<\e\},$$
where $\varrho$ is a fixed compatible metric on $M$ and
$\e:X\to(0,1]$ runs over continuous functions into $(0,1]$. The
symbol $\varrho(h,g)<\e$ means that
$\varrho\big(h(x),g(x)\big)<\e(x)$ for all $x\in X$. The source
limitation topology doesn't depend on the metric $\varrho$ \cite{nk}
and has the Baire property provided $M$ is completely metrizable
\cite{munkers}. Obviously, this topology coincides with the uniform
convergence topology when $X$ is compact.

We say that a space $M$ has the {\em $\NDDP[m]{n,k}$-property}
if any two
maps $f:\mathbb I^m\times \mathbb I^n\to M$, $g:\mathbb I^m\times
\mathbb I^k\to M$ can be approximated by maps $f':\mathbb I^m\times \mathbb
I^n\to M$ and $g':\mathbb I^m\times \mathbb I^k\to M$, respectively, such that
$f'(\{z\}\times \mathbb I^n)\cap g'(\{z\}\times \mathbb I^k)=\varnothing$ for all $z\in \mathbb I^m$.
Obviously, if $M$ has the $\NDDP[m]{n,k}$-property, then it also has the
$\NDDP[m']{n',k'}$-property for all $m'\leq m$, $n'\leq n$ and $k'\leq k$. The $\NDDP[0]{n,k}$-property
coincides with the well known disjoint $(n,k)$-cells property.
The $\NDDP[m]{n,k}$-property is very similar to the $\HDD[m]{n,k}$-property introduced in \cite[Definition 5.1]{bv}, where
it is required for any open cover $\U$ of $M$ the maps $f,g$ to be approximated by maps $f',g'$
such that $f',g'$ are $\U$-homotopic to $f$ and $g$, respectively and
$f'(\{z\}\times \mathbb I^n)\cap g'(\{z\}\times \mathbb I^k)=\varnothing$ for all $z\in \mathbb I^m$.
For example, it follows from \cite[Proposition 5.6 and Theorem 9.1]{bv} that every dendrite with a dense set of end-point has both
the $\NDDP[0]{0,\infty}$-property and the $\NDDP[1]{0,0}$-property, while $\mathbb R^{m+n+k+1}$ has the $\NDDP[m]{n,k}$-property.

The notion of continuum-wise injective maps was introduced in \cite{km} for maps between compact spaces.
Here we extend this definition for arbitrary spaces and arbitrary closed sets (not necessarily continua as in \cite{km}): A map $g:X\to M$ is {\em set-wise injective} if for
any two closed sets $A,B\subset X$ with $A\neq B$, we have $g(A)\neq g(B)$. We also consider the following
specialization of that property: a map $g:X\to M$ is {\em set-wise injective in dimension $k$} (see also \cite{km})
if $g(A)\neq g(B)$ for any two closed sets $A,B\subset X$ such that $\dim (A\setminus B )\geq k$. Obviously, every set-wise injective map in dimension 0 is injective.
Observed that for any two continua $A,B\subset X$ with $A\setminus B\neq\varnothing$ we have $\dim (A\setminus B )\geq 1$. Hence, every set-wise injective map in dimension $1$ is automatically
continuum-wise injective.

The main results in this paper is the following theorem, which is a parametric version of Theorem 3.11 from \cite{km} (recall that a map $f:X\to Y$ is $\sigma$-perfect if $X$ is a countable union of countably many closed sets $X_i$ such that each restriction $f|X_i:X_i\to f(X_i)$ is a perfect map):

\begin{thm} Let $f\colon X\to
Y$ be a $\sigma$-perfect surjective $n$-dimensional map between metric spaces such that $\dim Y\leq m$ and $M$ be a complete separable metric $\LC[2m+n-1]$-space with the
$\NDDP[m]{n,k}$-property with $k\leq n$. Then the function space $C(X,M)$ contains a dense
$G_\delta$-set of maps $g$ such that all restrictions $g|f^{-1}(y)$,
$y\in Y$, are set-wise injective in dimension $n-k$
\end{thm}

\begin{cor}
Let $X,Y$ and $f$ be as in Theorem $1.1$ and $P\subset Q\subset X$ be two $F_\sigma$-subsets of $X$ such that $\dim (P\cap f^{-1}(y))\leq p$ and
$\dim (Q\cap f^{-1}(y))\leq q$ for every $y\in Y$, where $0\leq p\leq q\leq n$. Then for every complete separable metric $\LC[2m+n-1]$-space $M$ with the $\NDDP[m]{q,p}$-property
the space $C(X,M)$ contains a dense $G_\delta$-set of maps $g$ satisfying the following condition: $g^{-1}(g(z))\cap Q\cap f^{-1}(y)=\{z\}$ for all $z\in P\cap f^{-1}(y)$ and all $y\in Y$.
\end{cor}

Note that when $p=q=n$ and $M$ having the $\HDD[m]{n,n}$-property, Corollary 1.2 was established in \cite[Theorem 3.3]{bv}. We already mentioned that
the space $\mathbb R^l$ has the $\NDDP[m]{n,k}$-property for all $m,n,k$ with $m+n+k<l$. Hence,
Corollary 1.2 is a far reaching generalization of Pasynkov's result \cite{pa2} stating that for any map $f:X\to Y$ between metrizable compacta the function space $C(X,\mathbb R^{\dim Y+2\dim f+1})$ contains a dense $G_\delta$-subset of maps that are injective on every fiber of $f$.

When $M$ is compact and $C(X,M)$ is equipped with the uniform convergence topology, analogues of Theorem 1.1 and Corollary 1.2 also hold.
Let us formulate the analogue of Theorem 1.1.
\begin{thm}
Let $M$ be a compact metric $\LC[2m+n-1]$-space with the
$\NDDP[m]{n,k}$-property with $k\leq n$ and $f\colon X\to Y$ be a closed surjective $n$-dimensional map between normal spaces such that $\dim Y\leq m$ and $W(f)\leq\aleph_0$. Then $C(X,M)$ equipped with the uniform convergence topology contains a dense subset of maps $g$ such that all restrictions $g|f^{-1}(y)$, $y\in Y$, are set-wise injective in dimension $n-k$.
\end{thm}

Recall that $W(f)\leq\aleph_0$ means that there exists a map $g:X\to\mathbb I^{\aleph_0}$ such that $f\triangle g$ embeds $X$ into $Y\times\mathbb I^{\aleph_0}$, see \cite{pa}. For example, according to \cite[Proposition 9.1]{pa}, $W(f)\leq\aleph_0$ for every closed map $f$ between metrizable spaces provided $f$ has Lindel\"{o}f fibers.

We apply Corollary 1.2 to provide a short proof of the following result:
\begin{pro}
Suppose $n,k$ are non-negative integers such that $k+1\leq n$. Then the product $M\times\mathbb R^{2l+1}$, where $l=n-k-1$, has the $\NDDP[m]{n,n}$-property
for every complete separable metric $\LC[2m+n-1]$-space $M$ with the $\NDDP[m]{n,k}$-property.
\end{pro}

The paper is organised as follows: all preliminary results are provided in Section 2, Section 3 contains the proofs of Theorem 1.1, Corollary 1.2 and Theorem 1.3. The proof of Proposition 1.4 is given in the Appendix.

\section{Some preliminary results}
In this section we suppose that the spaces $X,Y,M$ and the map $f\colon X\to Y$ satisfy the conditions from Theorem 1.1 with the additional assumption that the map $f$ is a perfect surjection.
Since $M$ is $\LC[2m+n-1]$, according to \cite[Lemma 4.1]{bv1}, $M$ admits a metric $\rho$
generating its topology and satisfying the following condition:
\begin{itemize}
\item[(1)] If $Z$ is an $(2m+n)$-dimensional metric space, $A\subset Z$ its closed set and
$h\colon Z\to M$ a map, then for every function $\alpha:Z\to(0,1]$ and
every map $g\colon A\to M$ with $\rho(g(z),h(z))<\alpha(z)/8$ for
all $z\in A$ there exists a map $\bar{g}\colon Z\to M$ extending $g$
such that $\rho(\bar{g}(z),h(z))<\alpha(z)$ for all $z\in Z$.
\end{itemize}

One can easily show that $(1)$ implies the following condition:
\begin{itemize}
\item[(2)] If $F\subset X$ is a closed set, the restriction map $\pi_F:C(X,M)\to C(F,M)$, $\pi_F(g)=g|F$, is an open and surjective map when
both $C(X,M)$ and $C(X,M)$ carry the source limitation topology.
\end{itemize}
For any set $K\subset Y$ and closed disjoint sets $A,B\subset X$  let denote by $C_K(X,M;A,B)$ the set of all $g\in C(X,M)$ such that:
\begin{itemize}
\item $g(A\cap f^{-1}(y))\cap g(B\cap f^{-1}(y))=\varnothing$ for every $y\in K$. If $K=Y$, we write $C(X,M;A,B)$ instead of $C_K(X,M;A,B)$.
\end{itemize}


The aim of this section is to show that all sets $C_K(X,M;A,B)$ are open and dense in $C(X,M)$ with respect to
the source limitation topology. Our proofs are based on some ideas from \cite{mv} and \cite{v3}.

\begin{lem}
Let $K\subset Y$ be closed and $g_0\in C_K(X,M;A,B)$, where $A,B$ are disjoint closed subsets of $X$. Then there is a continuous function $\alpha:X\to (0,1]$ and an open set
$W\subset Y$ containing $K$ such that $g\in C_W(X,M;A,B)$ provided $g\in C(X,M)$ and $\rho(g(x),g_0(x))<\alpha(x)$ for all $x\in X$.
\end{lem}

\begin{proof} One can show that for every $y\in K$ there exists a neighborhood $V_y\subset Y$ of $y$ and a positive
number $\delta_y\leq 1$ such that $y'\in V_y$ and $\rho(g(x),g_0(x))<\delta_y$ for all $x\in f^{-1}(y')$, where $g\in C(X,M)$, yields $g\in C_{y'}(X,M;A,B)$.
Let $V=\bigcup_{y\in K}V_y$ and $W\subset Y$ be an open set containing $K$ with $\overline W\subset V$.
The family
$\{V_y:y\in K\}$ can be supposed to be locally finite in $V$.
Consider the set-valued lower semi-continuous map $\varphi\colon\overline W\to (0,1]$, $\varphi(y)=\cup\{(0,\delta_z]:y\in V_z\}$. By
\cite[Theorem 6.2, p.116]{rs}, $\varphi$ admits a continuous
selection $\beta\colon\overline W\to (0,1]$. Let $\overline{\beta}:Y\to
(0,1]$ be a continuous extension of $\beta$ and
$\alpha=\overline{\beta}\circ f$. The set $W$ is the required one.
\end{proof}

\begin{cor}
Each set $C_K(X,M;A,B)$ is open in $C(X,M)$.
\end{cor}

\begin{proof}
Let $g_0\in C_K(X,M;A,B)$. By Lemma 2.1, there exists a function $\alpha:X\to (0,1]$ such that $g\in C_K(X,M;A,B)$ for any $g\in C(X,M)$
satisfying the inequality $\rho(g(x),g_0(x))<\alpha(x)$ for all $x\in X$. Then $B_\rho(g_0,\alpha)$ is a neighborhood of $g_0$ and
$B_\rho(g_0,\alpha)\subset C_K(X,M;A,B)$. 
\end{proof}

Next step is to show that if $K\subset Y$ is closed, $A$ and $B$ are disjoint closed subsets of $X$ with $\dim f|A\leq k$, then $C_K(X,M;A,B)$ is dense in $C(X,M)$. To this end we need some preliminary results. The first one is the following characterization of spaces with the $\NDDP[m]{n,k}$-property, which can be obtain from the proof
of \cite[Theorem 5.7]{bv}:
\begin{pro} Let $m,n,k$ be non-negative integers and $d=m+\max\{n,k\}$.
A Polish
$\LC[d-1]$-space $M$ has the $\NDDP[m]{n,k}$-property if
and only if for any separable polyhedron $P$
with $\dim P\le m$ there are two disjoint
$\sigma$-compact sets $E_n,E_k\subset P\times M$
such that $E_n\in\MAP[P]{n}$ and $E_k\in\MAP[P]{k}$.
\end{pro}

The notation $E_n\in\MAP[P]{n}$ means that
for any
$n$-dimensional map $p:K\to P$ with $K$ being a finite-dimensional
metric compactum, a closed subset $F\subset K$, a map $g:K\to M$, and
a positive $\delta$ there is a map $g':K\to M$ such that $g'$
is $\delta$-close to $g$, $g'|F=g|F$ and $(p\triangle g')(K\setminus
F)\subset E_n$.

To prove the density of the sets $C_K(X,M;A,B)$, where $A,B$ are disjoint closed subsets of $X$ with $\dim f|A\leq k$, we fix a map $g_0:X\to M$ and a function $\e:X\to (0,1]$. Define the set-valued map
$$\Phi_\e:Y\to 2^{C(X,M)}{~}\mbox{by}{~}\Phi_\e(y)=B_\rho(g_0,\e)\cap C_y(X,M;A,B),$$ where $C(X,M)$ carries the compact-open topology.

\begin{lem} All $\Phi_\e(y)$ are non-empty sets. Moreover,
if $\Phi_\e(y_0)$ contains a compact set $K$ for some $y_0\in Y$,
then there exists a neighborhood $V(y_0)$ of $y_0$ such that
$K\subset\Phi_\e(y)$ for every $y\in V(y_0)$.
\end{lem}

\begin{proof}
Since $M$ is an $\LC[n-1]$-space with the disjoint $(n,k)$-cells property and $\dim f^{-1}(y)\cap A\leq k$, the set of all maps $h\in C(f^{-1}(y),M)$ with
$h(A\cap f^{-1}(y))\cap h(B\cap f^{-1}(y))=\varnothing$
is dense in $C(f^{-1}(y),M)$ (see the proof of Lemma 3.4 from \cite{km}). So, if $\delta_y=\min\{\e(x):x\in f^{-1}(y)\}$, then there exists such a map
$h\in C(f^{-1}(y),M)$ with $\rho(h,g_0|f^{-1}(y))<\delta_y/8$. Then, by the extension property $(1)$, $h$ can
be extended to a map $g\in C(X,M)$ such that $\rho(g,g_0)<\e$. Obviously $g\in C_y(X,M;A,B)$, so
$\Phi(y)\neq\varnothing$ for all $y\in Y$.

The second part of that lemma can be established following the proof of Lemma 2.5(2) from \cite{mv}.
\end{proof}

\begin{lem} Every
$\Phi_\e(y)$ has the following property: If
$\hat{v}\colon\mathbb{S}^p\to\Phi_\e(y)$ is continuous, where $p\leq m-1$
and $\mathbb{S}^p$ is the $p$-sphere, then $\hat{v}$ can be
extended to a continuous map
$\displaystyle\hat{u}\colon\mathbb{I}^{p+1}\to\Phi_{16\e}(y)$.
\end{lem}

\begin{proof} Let us mention
the following property of the function space $C(X,M)$ with the
compact open topology: For any metrizable space $Z$ a map
$\hat{w}\colon Z\to C(X,M)$ is continuous if and only if the map
$w\colon Z\times X\to M$, $w(z,x)=\hat{w}(z)(x)$, is continuous.
Hence, every map $\hat{v}\colon\mathbb{S}^p\to\Phi_\e(y)$ generates
a continuous map $v\colon\mathbb{S}^p\times X\to M$ defined by
$v(z,x)=\hat{v}(z)(x)$ such that
$\rho\big(v(z,x),g_0(x)\big)<\e(x)$ for all
$(t,x)\in\mathbb{S}^p\times X$.

Define the maps $\overline{g}_0:\mathbb I^{p+1}\times X\to M$ and $\overline{\e}:\mathbb I^{p+1}\times X\to (0,1]$
by $\overline{g}_0(t,x)=g_0(x)$ and $\overline{\e}(t,x)=\e(x)$ for all $t\in\mathbb I^{p+1}$.
Since $X$ admits a perfect $n$-dimensional map onto the $m$-dimensional space $Y$, $\dim X\leq n+m$, see \cite{re:95}.
Hence, $\dim(\mathbb I^{p+1}\times X)\leq 2m+n$. Then, according to the extension property $(1)$, $v$ can be extended to a map
$v_1:\mathbb{I}^{p+1}\times X\to M$ such that $\rho(v_1,\overline{g}_0)<8\overline{\e}$. Let $A_y=A\cap f^{-1}(y)$, $B_y=B\cap f^{-1}(y)$.
Denote
by $v_{1,A}:\mathbb I^{p+1}\times A_y\to M$ and  $v_{1,B}:\mathbb I^{p+1}\times B_y\to M$, respectively,
the restrictions $v_1|(\mathbb I^{p+1}\times A_y)$ and $v_1|(\mathbb I^{p+1}\times B_y)$.
By Proposition 2.3, there exist two disjoint subsets $E_{k}$ and $E_n$ of $\mathbb I^{p+1}\times M$ such that $E_n\in\MAP[\mathbb I^{p+1}]{n}$ and $E_{k}\in\MAP[\mathbb I^{p+1}]{k}$. Applying the $(\MAP[\mathbb I^{p+1}]{k})$-property of $E_{k}$ with respect to the projection $\pi_A:\mathbb I^{p+1}\times A_y\to\mathbb I^{p+1}$, we find a map $h_A:\mathbb I^{p+1}\times A_y\to M$ satisfying the following conditions,
where $\delta_y=\min\{8\e(x)-\rho(v_1(t,x),g_0(x)):(t,x)\in\mathbb I^{p+1}\times f^{-1}(y)\}$:
\begin{itemize}
\item[(3)] $h_A|(\mathbb S^{p}\times A_y)=v_1|(\mathbb S^{p}\times A_y)$;
\item[(4)] $\rho(h_A,v_{1,A})<\delta_y$;
\item[(5)] $\pi_A\triangle h_A((\mathbb I^{p+1}\setminus\mathbb S^p)\times A_y)\subset E_{k}$.
\end{itemize}

Applying the $(\MAP[\mathbb I^{p+1}]{n})$-property of $E_{n}$ with respect to the projection $\pi_B:\mathbb I^{p+1}\times B_y\to\mathbb I^{p+1}$, we obtain a map $h_B:\mathbb I^{p+1}\times B_y\to M$ such that
\begin{itemize}
\item[(6)] $h_B|(\mathbb S^{p}\times B_y)=v_1|(\mathbb S^{p}\times B_y)$;
\item[(7)] $\rho(h_B,v_{1,B})<\delta_y$;
\item[(8)] $\pi_B\triangle h_B((\mathbb I^{p+1}\setminus\mathbb S^p)\times B_y)\subset E_{n}$.
\end{itemize}
Consider now the map $h:F\to M$, where $F=(\mathbb S^p\times X)\cup (\mathbb I^{p+1}\times A_y)\cup(\mathbb I^{p+1}\times B_y)$, such that
$h|(\mathbb S^p\times X)=v_1|(\mathbb S^p\times X)$, $h|(\mathbb I^{p+1}\times A_y)=h_A$ and  $h|(\mathbb I^{p+1}\times B_y)=h_B$.
Observed that $\rho(h(t,x),v_1(t,x))<\e(x)$ for all $(t,x)\in F$. So, using again the extension property $(1)$, we extend the map $h$ to a map
$\widetilde h:\mathbb I^{p+1}\times X\to M$ with $\rho(\widetilde h,v_1)<8\overline{\e}$. Because $\rho(v_1,\overline{g}_0)<8\overline{\e}$, we
have $\rho(\widetilde h,\overline{g}_0)<16\overline{\e}$. Then $\widetilde h$ provides a map $\hat{u}:\mathbb I^{p+1}\to C(X,M)$, defined by
$\hat{u}(t)(x)=\widetilde h(t,x)$, such that $\hat{u}(t)\in B_\rho(g_0,16\e)$ for all $t\in \mathbb I^{p+1}$.

It remains to show that $\hat{u}(\mathbb I^{p+1})\subset\Phi_{16\e}(y)$. To this end, observe that conditions
$(5)$ and $(8)$ imply $\widetilde h(\{t\}\times A_y)\cap\widetilde h(\{t\}\times B_y)=\varnothing$ for all $t\in\mathbb I^{p+1}\setminus\mathbb S^{p}$.
Because $\widetilde h|(\mathbb S^{p}\times f^{-1}(y))=v|(\mathbb S^{p}\times f^{-1}(y))$ and $\hat{v}(t)\in C_y(X,M;A,B)$,
$\widetilde h(\{t\}\times A_y)\cap\widetilde h(\{y\}\times B_y)=\varnothing$ for any $t\in\mathbb S^{p}$. Therefore,
$\widetilde h(\{t\}\times A_y)\cap\widetilde h(\{t\}\times B_y)=\varnothing$ for all $t\in\mathbb I^{p+1}$. The last condition yields
$\hat{u}(\mathbb I^{p+1})\subset C_y(X,M;A,B)$. Hence, $\hat{u}(\mathbb I^{p+1})\subset\Phi_{16\e}(y)$.
\end{proof}

\begin{pro} $C_K(X,M;A,B)$ is a dense subset of $C(X,M)$
with respect to the source limitation topology for every closed $K\subset Y$.
\end{pro}

\begin{proof} Define the set-valued maps $\Phi_i\colon K\to
C(X,M)$, $i=0,..,m$,
$\displaystyle\Phi_i(y)=\Phi_{\e/16^{m-i+1}}(y)$. Obviously,
$\Phi_0(y)\subset\Phi_1(y)\subset...\subset\Phi_m(y)=\Phi_{\e/16}(y)$.
According to Lemma 2.5, every map from $\mathbb{S}^p$ into
$\Phi_i(y)$ can be extended to a map from $\mathbb{I}^{p+1}$ into
$\Phi_{i+1}(y)$, where $p\leq m-1$, $i=0,1,..,m-1$ and $y\in K$. Moreover, by
Lemma 2.4, any $\Phi_i(y)$ has the following property: if
$P\subset\Phi_i(y)$ is compact, then there exists a neighborhood
$V_y$ of $y$ in $Y$ such that $P\subset\Phi_i(z)$ for all $z\in
V_y\cap K$. So, we may apply the proof of \cite[Theorem 3.1]{gu} to find a
continuous selection $\theta\colon K\to C(X,M)$ of $\Phi_m$. Hence,
$\theta(y)\in\Phi_{\e/16}(y)$ for all $y\in K$. Now, consider the map
$g\colon f^{-1}(K)\to M$, $g(x)=\theta(f(x))(x)$. Using that
$C(X,M)$ carries the compact open topology, one can show that $g$ is
continuous. Moreover, $\varrho\big(g(x),g_0(x)\big)<\e(x)/16$ for all
$x\in f^{-1}(K)$. Then, by $(1)$, $g$ can be extended to a
continuous map $\bar{g}\colon X\to M$ with
$\varrho\big(\bar{g}(x),g_0(x)\big)<\e(x)$, $x\in X$. It follows
from the definition of $g$ that $g|f^{-1}(y)=\theta(y)|f^{-1}(y)$
for every $y\in K$. Since $\theta(y)\in C_y(X,M;A,B)$, $\bar{g}(A_y)\cap\bar{g}(B_y)=\varnothing$ for all
$y\in K$. Hence, $\bar{g}\in B_\varrho(g_0,\e)\cap C_K(X,M;A,B)$.
\end{proof}

\section{Proofs}
\textit{Proof of Theorem $1.1$.} Let $X$ be the union of an increasing sequence $\{X_i\}_{i\geq 1}$ of closed sets such that each restriction $f_i=f|X_i$ is a perfect map.
So, according to condition $(2)$, the restriction maps $\pi_i:C(X,M)\to C(X_i,M)$ are open surjections when both $C(X,M)$ and $C(X_i,M)$ are equipped with
the source limitation topology. Hence, by Corollary 2.2 and Proposition 2.6, the sets $\pi_i^{-1}(C(X_i,M;A\cap X_i,B\cap X_i;f_i))$ are open and dense in
$C(X,M)$ for any $i$, where $A$ and $B$ are closed disjoint subsets of $X$ with $\dim f|A\leq k$. Here, $C(X_i,M;A\cap X_i,B\cap X_i;f_i)$ is the set of all
$g\in C(X_i;M)$ such that $g(A\cap f_i^{-1}(y))\cap g(B\cap f_i^{-1}(y))=\varnothing$ for all $y\in f_i(X_i)$. Similarly, $C(X,M;A,B;f)$ denotes the set
of the maps $g\in C(X,M)$ with $g(A\cap f^{-1}(y))\cap g(B\cap f^{-1}(y))=\varnothing$ for all $y\in Y$.
Since
$$C(X,M;A,B;f)=\bigcap_{i=1}^{\infty}\pi_i^{-1}(C(X_i,M;A\cap X_i,B\cap X_i;f_i)),$$  any $C(X,M;A,B;f)$ is a dense $G_\delta$-subset of $C(X,M)$.

Suppose first that $k\leq n-1$. Since $f$ is $\sigma$-perfect and $\dim f\leq n$, there exist closed subsets $F_i\subset X$, $i=1,2,..$, such that $\dim F_i\leq k$ for each $i$ and
the restriction $f|(X\setminus\bigcup_{i=1}^\infty F_i)$ is a map of dimension $\leq n-k-1$, see \cite[Theorem 1.4]{tv}. Because each $f_i$ is a perfect map, by \cite[Proposition 9.1]{pa}, there exist maps $h_i\colon X_i\to\mathbb I^{\aleph_0}$ embedding all fibers of $f_i$, $i\geq 1$. We can suppose that each $h_i$ is defined on $X$. Hence, the diagonal product $h$ of all $h_i$ is a map from $X$ into $\mathbb I^{\aleph_0}$ such that  $h|f^{-1}(y):f^{-1}(y)\to\mathbb I^{\aleph_0}$ is one-to one for all $y\in Y$. We fix a finitely additive base $\Gamma=\{U_j\}_{j\geq 1}$ for the topology of $\mathbb I^{\aleph_0}$ and consider the family $\mathcal A$ of all non-empty intersections
$h^{-1}(\overline U_j)\cap F_i$, $i,j=1,2,..$, and the family $\mathcal B=\{h^{-1}(\overline U_j)\}_{j\geq 1}$. Obviously, $\dim A\leq k$ for all $A\in\mathcal A$. We already observed that the sets
$C(X,M;A,B;f)$,  where $A\in\mathcal A$ and $B\in\mathcal B$ are disjoint, are dense and $G_\delta$ in $C(X,M)$ with respect to the source limitation topology.
Then the intersection $\mathcal S$ of all $C(X,M;A,B;f)$ is also a dense $G_\delta$-subset of $C(X,M)$.

Let us show that $\mathcal S$ consists of maps $g$ such that each restriction $g|f^{-1}(y)$, $y\in Y$, is set-wise injective in dimension $n-k$. Indeed, suppose $K_1\neq K_2$ are two non-trivial closed sets, which are contained in some $f^{-1}(y_0)$ and $\dim (K_2\setminus K_1)\geq n-k$.

\textit{Claim $1$. There is $x_0\in (K_2\setminus K_1)\cap (\bigcup_{i=1}^\infty F_i)$.}

Indeed, otherwise $K_2\setminus K_1\subset f^{-1}(y_0)\setminus(\bigcup_{i=1}^\infty F_i)$, which implies $\dim K_2\setminus K_1\leq n-k-1$, a contradiction.

Next claim completes the proof of Theorem 1.1 in the case $k\leq n-1$.

\textit{Claim $2$. $g(x_0)\not\in g(K_1)$ for all $g\in\mathcal S$.}

We fix $i_0$ with $x_0\in F_{i_0}$. Since $h(x_0)\in h(K_2)\setminus h(K_1\cap X_i)$ and $h(K_1\cap X_i)$ is a compact set for every $i$, there exist $U_{j_i}, U_{l_i}\in\Gamma$ such that $h(x_0)\in U_{j_i}$, $h(K_1\cap X_i)\subset U_{l_i}$ and
$\overline{U}_{j_i}\cap\overline U_{l_i}=\varnothing$ (recall that $\Gamma$ is finitely additive). Then $h^{-1}(\overline U_{j_i})$ and $B_i=h^{-1}(\overline U_{l_i})$ are also disjoint and
$K_1\cap X_i\subset B_i\cap f^{-1}(y_0)$. Moreover
$A_i=h^{-1}(\overline U_{j_i})\cap F_{i_0}\in\mathcal A$ and $x_0\in A_i$.
Consequently, $g(x_0)\not\in g(K_1\cap X_i)$ for all $g\in C(X,M;A_i,B_i;f)$ and all $i$. Finally, since $g(K_1)=\bigcup_{i=1}^\infty g(K_1\cap X_i)$, we have
$g(x_0)\in g(K_2)\setminus g(K_1)$.

Suppose now that $k=n$, and let $\Gamma=\{U_j\}_{j\geq 1}$ and $\mathcal B$ be as above. Then the intersection of all $C(X,M;A,B;f)$, where
$A,B\in\mathcal B$ are disjoint, is a dense $G_\delta$-subset of $C(X,M)$ and consists of maps $g$ such that the restrictions $g|f^{-1}(y)$, $y\in Y$, are set-wise injective in dimension $0$.
\hfill$\square$

\smallskip
\textit{Proof of Corollary $1.2$.} Suppose first that $Q\subset X$ is closed, and let $f_Q=f|Q$ and $Y_Q=f(Q)$. Obviously, $f_Q:Q\to Y_Q$ is a $\sigma$-perfect surjection with $\dim f_Q\leq q$. Then, we apply Theorem 1.1 (with $X, Y, f$ replaced, respectively, by $Q, Y_Q, f_Q$) to show the existence of a dense $G_\delta$-subset of $C(Q;M)$ of maps $g$ such that all restrictions $g|f_Q^{-1}(y)$, $y\in Y_Q$, are set-wise injective in dimension $q-p$. More precisely, following the notations from the proof of Theorem 1.1, we find countably many disjoint couples $(A_i,B_i)$ of closed subsets of $X$ satisfying the following conditions:
\begin{itemize}
\item $A_i,B_i\subset Q$;
\item Each $C(Q,M;A_i,B_i;f_Q)$ is a dense $G_\delta$-subset of $C(Q,M)$ and the intersection $\mathcal S_Q$ of all $C(Q,M;A_i,B_i;f_Q)$ consists of maps $g\in C(Q,M)$ such that $g|f_Q^{-1}(y)$, $y\in Y_Q$, is set-wise injective in dimension $q-p$;
\item If $p\leq q-1$, then for any $y\in Y_Q$ and any two different points $z\in f_Q^{-1}(y)\cap P$ and $x\in f_Q^{-1}(y)$ there exists a couple $(A_i,B_i)$ with $z\in A_i$ and $x\in B_i$;
 \item If $p=q$, then the couples $(A_i,B_i)$ are separating the points of $f_Q^{-1}(y)$ for all $y\in Y_Q$.
\end{itemize}
The last two properties yield that $\mathcal S_Q$ consists of maps $g\in C(Q,M)$ such that $g^{-1}(g(z))\cap f_Q^{-1}(y)=\{z\}$ for all $z\in P\cap f_Q^{-1}(y)$ and all $y\in Y_Q$. Let $\pi_Q:C(X,M)\to C(Q,M)$ be the restriction map. According to condition $(2)$, each set $\pi_Q^{-1}(C(Q,M;A_i,B_i;f_Q))$ is dense and
$G_\delta$ in $C(X,M)$. Then the set $\pi_Q^{-1}(\mathcal S_Q)$ is also dense and $G_\delta$ in $C(X,M)$, and consists of
maps $g$ such that $g^{-1}(g(z))\cap Q\cap f^{-1}(y)=\{z\}$ for all $z\in P\cap f^{-1}(y)$ and all $y\in Y$.

If $Q=\bigcup_{j=1}^\infty Q_j$ is an $F_\sigma$-subset of $X$, we consider the $\sigma$-perfect restrictions $f_j=f|Q_j$ and the spaces $Y_j=f_j(Q_j)$. As above, for each $j$ we find countably many couples $(A_i^j,B_i^j)$ of closed disjoint subsets of $Q_j$ such that the intersection
$\mathcal S_{Q_j}$ of all $C(Q_j,M;A_i^j,B_i^j;f_{j})$, $i\geq 1$, is dense and $G_\delta$ in $C(Q_j;M)$. Consequently,
$\mathcal S=\bigcap_{j=1}^\infty\pi_{Q_j}^{-1}(\mathcal S_{Q_j})$
is dense and $G_\delta$ in $C(X,M)$. It is easily seen that any $g\in\mathcal S$ satisfies the required condition
$g^{-1}(g(z))\cap Q\cap f^{-1}(y)=\{z\}$ for all $z\in P\cap f^{-1}(y)$ and all $y\in Y$.
\hfill$\square$

\smallskip
\textit{Proof of Theorem $1.3$.} We follow the approach from the proof of \cite[Theorem 1.2]{v2}.
Fix a map $g_0\colon X\to M$ and a number
$\epsilon>0$.
Since $W(f)\leq\aleph_0$, there
exists a map $\lambda\colon X\to\mathbb I^{\aleph_0}$ such that
$f\triangle\lambda\colon X\to Y\times\mathbb I^{\aleph_0}$ is an
embedding.  We are going to find a map $g\in C(X,M)$ such that $g$
is $\epsilon$-close to $g_0$ and all restrictions $g|f^{-1}(y)$,
$y\in Y$, are continuum-wise injective in dimension $n-k$. To this end, let
$\overline{\lambda}\colon\beta X\to\mathbb I^{\aleph_0}$,
$\overline{g}_0\colon\beta X\to M$ and $\overline f:\beta X\to\beta Y$ be the Stone-Cech extensions of
the maps $\lambda$, $g_0$ and $f$, respectively. Then
$\overline{\lambda}\triangle\overline{g}_0\in C(\beta X,\mathbb
I^{\aleph_0}\times M)$. We consider also the constant maps
$h\colon\mathbb I^{\aleph_0}\times M\to Pt$ and
$\eta\colon\beta Y\to Pt$, where $Pt$ is the one-point space.
According to Pasynkov's factorization theorem \cite[Theorem
13]{pa1}, there exist metrizable compacta $K$, $T$ and  maps
$f_{*}\colon K\to T$, $\xi_1\colon\beta X\to K$, $\xi_2\colon K\to
\mathbb I^{\aleph_0}\times M$ and $\eta_1\colon\beta Y\to T$ such
that:
\begin{itemize}
\item $\eta_1\circ\overline f=f_{*}\circ\xi_1$;
\item $\xi_2\circ\xi_1=\overline{\lambda}\triangle\overline{g}_0$;
\item $\dim T\leq\dim\beta Y$ and $\dim f_*\leq\dim\overline f$.
\end{itemize}
Since $Y$ is normal, $\dim\beta Y=\dim Y\leq m$. Moreover, by \cite[Proposition 8]{pa1}, $\dim f\leq n$ implies $\dim\overline f\leq n$.
If $p\colon\mathbb I^{\aleph_0}\times M\to\mathbb I^{\aleph_0}$ and
$q\colon\mathbb I^{\aleph_0}\times M\to M$ denote the corresponding
projections, we have
\begin{itemize}
\item $p\circ\xi_2\circ\xi_1=\overline{\lambda}\hbox{~}\mbox{and}\hbox{~}q\circ\xi_2\circ\xi_1=\overline{g}_0.$
\end{itemize}
By Theorem 1.1, there exists a map $\phi\colon K\to
M$ such that $\phi$ is $\epsilon$-close to $q\circ\xi_2$ and all restrictions $\phi|f_*^{-1}(t)$, $t\in T$, are
set-wise injective in dimension $n-k$. Then the
map $\overline{g}=\phi\circ\xi_1$ is $\epsilon$-close to
$\overline{g}_0$. Hence, the maps $g=\overline{g}|X$ and $g_0$ are
also $\epsilon$-close. Because
$\lambda=\big(p\circ\xi_2\circ\xi_1\big)|X$ embeds the fibers of $f$ into $\mathbb I^{\aleph_0}$,
$\xi_1$ embeds the fibers of $f$ into $K$ such that $f^{-1}(y)\subset f_*^{-1}(\eta_1(y))$ for all $y\in Y$. Therefore, the restrictions
$g|f^{-1}(y)$, $y\in Y$, are set-wise injective in dimension $n-k$. \hfill$\square$

\section{Appendix}
\textit{Proof of Proposition $1.4$.}
Let $f:\mathbb I^m\times\mathbb I^n\to M\times\mathbb R^{2l+1}$ and $g:\mathbb I^m\times\mathbb I^{n}\to M\times\mathbb R^{2l+1}$ be two maps. We are going to approximate $f$ and $g$ by maps $f':\mathbb I^m\times\mathbb I^n\to M\times\mathbb R^{2l+1}$ and $g':\mathbb I^m\times\mathbb I^{n}\to M\times\mathbb R^{2l+1}$ such that $f'(\{z\}\times\mathbb I^n)\cap g'(\{z\}\times \mathbb I^{n})=\varnothing$ for all $z\in \mathbb I^m$. To this end,
let $\varphi:\mathbb I^m\times(\mathbb I^n\oplus\mathbb I^{n})\to M\times\mathbb R^{2l+1}$ be the map generated by $f$ and $g$, where $\oplus$ denotes the discrete sum. Represent $\varphi$ as the product
$\varphi=\varphi_1\times\varphi_2$ of two maps $\varphi_1:\mathbb I^m\times(\mathbb I^n\oplus\mathbb I^{n})\to M$ and
$\varphi_2:\mathbb I^m\times(\mathbb I^n\oplus\mathbb I^{n})\to\mathbb R^{2l+1}$.

\textit{Claim $3$. There exists an $F_\sigma$-subset $F\subset\mathbb I^m\times(\mathbb I^n\oplus\mathbb I^{n})$ such that $\dim (X\setminus F)\leq n-k-1$ and $\dim\pi|F\leq k$, where $\pi:\mathbb I^m\times(\mathbb I^n\oplus\mathbb I^{n})\to\mathbb I^m$ is the projection.}

Indeed, denote $X=\mathbb I^m\times(\mathbb I^n\oplus\mathbb I^{n})$ and take an $F_\sigma$-set $H\subset X$ such that $\dim H\leq k$ and
$\dim\pi|(X\setminus H)\leq n-k-1$, see \cite{to}. Then $H$ is contained in a $G_\delta$-set $\widetilde H\subset X$ with $\dim\tilde H\leq k$, and the set $F=X\setminus\widetilde H$ is the required one.

Since $\dim\pi|F\leq k$ and $M$ has the $\NDDP[m]{n,k}$-property, by Corollary 1.2, there exists a map $\phi_1:\mathbb I^m\times(\mathbb I^n\oplus\mathbb I^{n})\to M$ sufficiently close to $\varphi_1$ with $\phi_1^{-1}(\phi_1(z))\cap\pi^{-1}(y)=\{z\}$ for all $z\in F\cap\pi^{-1}(y)$ and all $y\in\mathbb I^m$. Let $\widetilde F$ be the set of all $z\in X$ such that $\phi_1^{-1}(\phi_1(z))\cap\pi^{-1}(y)=\{z\}$ for all $y\in\mathbb I^m$. It is easily seen that $\widetilde F=\{z\in X:(\pi\triangle\phi_1)^{-1}((\pi\triangle\phi_1)(z))=\{z\}\}$, where $\pi\triangle\phi_1:X\to\mathbb I^m\times M$ is the diagonal product of $\pi$ and $\phi_1$. Because $\pi\triangle\phi_1$ is a closed map, $\widetilde F$ is a $G_\delta$-subset of $X$. Moreover, $\widetilde F$ contains $F$ and $P=X\setminus\widetilde F$ is an $\sigma$-compact set of dimension $\dim P\leq l$. Thus, there is a map $\phi_2:X\to\mathbb R^{2l+1}$ sufficiently close to $\varphi_2$ such that
$\phi_2|P$ is one-to-one. Then the map $\phi_1\times\phi_2:\mathbb I^m\times(\mathbb I^n\oplus\mathbb I^{n})\to M\times\mathbb R^{2l+1}$ is close to $\varphi$. Consequently, the maps
$f'=\phi|(\mathbb I^m\times\mathbb I^n)$ and $g'=\phi|(\mathbb I^m\times\mathbb I^{n})$ are close, respectively,
to $f$ and $g$. Moreover, $f'(\{z\}\times\mathbb I^n)\cap g'(\{z\}\times \mathbb I^{n})=\varnothing$ for all $z\in \mathbb I^m$. \hfill$\square$

\textbf{Acknowledgments.} The second author was partially supported by NSERC Grant 261914-03. The paper was finalized during his visit to Shimane University in April 2016. He appreciates the hospitality of his colleagues at Shimane university.


\begin{thebibliography}{999}

\bibitem{bv} T.~Banakh and V.~Valov, \textit{General
position properties in fiberwise geometric topology}, Dissert. Math. 491, Warszawa 2013.

\bibitem{bv1} T.~Banakh and V.~Valov, \textit{Approximation by light maps and parametric Lelek maps},
Topology and Appl. \textbf{157} (2010), 2325--2341.




\bibitem{re:95} R.~Engelking, \textit{Theory of
dimensions: Finite and Infinite}, Heldermann Verlag, Lemgo (1995).


\bibitem{gu} V.~Gutev, \textit{Selections and approximations in finite-dimensional spaces},
Topology and Appl. \textbf{146/147} (2005), 353--383.

\bibitem{km}
H.~Kato and E.~Matsuhashi, \textit{Continuum-wise injective maps}, Topology and Appl. \textbf{202}, (2016), 410--417.

\bibitem{mv}
E.~Matsuhashi and V.~Valov, \textit{Krasinkiewicz spaces and parametric Krasinkiewicz maps},
Houston J. Math. \textbf{36} (2010), no. 4, 1207--1220.



\bibitem{nk} N.~Krikorian, \textit{A note concerning the
fine topology on function spaces}, Compos. Math. \textbf{21} (1969),
343--348.





\bibitem{munkers} J.~Munkers, \textit{Topology} (Prentice Hall, Englewood Cliffs, NY, 1975).

\bibitem{pa}
B.~Pasynkov, {\em On geometry of continuous maps of finite functional weight}, Fund. Prikl. Mat.
{\bf 4} (1998), 155--164.

\bibitem{pa2}
B.~Pasynkov, {\em On geometry of continuous maps of finite compact metric spaces}, Proc. Steklov Inst. Math.
{\bf 212} (1996), 138--162

\bibitem{pa1} B.~Pasynkov, {\em Factorization theorems in dimension theory}, Uspekhi Mat. Nauk {\bf 36} (1981), no.3, 147--175.

\bibitem{rs} D.~Repov\v{s} and P.~Semenov, \textit{Continuous
selections of multivalued mappings}, Math. and its Appl.
\textbf{455}, Kluwer, Dordrecht (1998).

\bibitem{to}
H.~Torunczyk, \textit{Finite-to-one restrictions of continuous functions}, Fund. Math. \textbf{125} (1985), 237--249.

\bibitem{tv}
M.~Tuncali and V.~Valov, \textit{On dimensionally restricted maps}, Fund. Math. \textbf{175} (2002), 35--52.



\bibitem{v2}
V.~Valov, \textit{Another approach to parametric Bing and Krasinkiewicz maps}, Math. Balkanica \textbf{25} (2011), no. 4, 419--423.

\bibitem{v3}
V.~Valov, \textit{Parametric Bing and Krasinkiewicz maps}, Topology and Appl. \textbf{155} (2008), no. 8, 906--915.

\bibitem{w}
H.~Whitney, \textit{Differential manifolds}, Ann. Math. \textbf{37}
(1936), 645--680.

\end{thebibliography}
\end{document}